\documentclass[preprint,12pt]{elsarticle}

\usepackage{enumerate}
\usepackage{amssymb}
\usepackage{lipsum}
\usepackage[a4paper, total={6.9in, 9.4in}]{geometry}
\usepackage{booktabs}
\usepackage{amsmath,amssymb,url}
\usepackage{enumitem} 
\usepackage{graphics} 
\usepackage{array}
\usepackage{dsfont}
\usepackage[all]{xy}
\usepackage{amsthm}
\usepackage{tikz}
\usepackage{bbding}
\usepackage{graphicx}
\usepackage{amsthm}
\usepackage{mathrsfs}

\numberwithin{equation}{section}
\newtheorem{theorem}{Theorem}[section]
\newtheorem{lemma}[theorem]{Lemma}

\newcommand{\dn}{\mathord{\downarrow}\hspace{0.05em}}
\newcommand{\up}{\mathord{\uparrow}\hspace{0.05em}}
\newcommand{\uuar}{\mathord{\Downarrow}\hspace{0.05em}}
\newcommand{\dda}{\mathord{\Uparrow}\hspace{0.05em}}
\newcommand\blfootnote[1]{%
\begingroup
\renewcommand\thefootnote{}\footnote{#1}%
\addtocounter{footnote}{-1}%
\endgroup
}

\journal{Fundamenta Mathematicae}

\begin{document}

\begin{frontmatter}



\title{The sobriety for locally compact dcpos}


\author{Zhengmao He$^\star$, Yujing Zhang$^*$}
\address{$^\star$School of Sciences, Southwest Petroleum University, Chengdu 610500, Sichuan, China}
\address{$^*$School of Sciences, Hebei University of Science and Technology, Shijiazhuang 05118, Hebei, China}

\begin{abstract} In this paper, we focus on an open problem whether every locally compact dcpo is sober for the Scott topology. Specifically, we obtain two sufficient conditions for answering this problem. The main results are

(1) every meet continuous locally compact coherent dcpo is sober with respect to the Scott topology;

(2) the Scott space of every locally compact second countable dcpo is a sober space.
\end{abstract}
\begin{keyword}sober space; Scott topology; locally compact dcpo; second countable; meet continuous \\
\vspace*{0.2cm}
\end{keyword}


\end{frontmatter}
\blfootnote{$^\star$ Corresponding author: Zhengmao He}
\blfootnote{This work is supported by the National Natural Science Foundation of China (Grant nos.12501636, 12471438, 12601906) and the Sichuan Science and Technology Program (Grant No. 2026NSFSC0784).}
\blfootnote{E-mail address: hezhengmaomath@163.com (Z.M.He);zhangyujing2024@163.com.}


\section{Introduction}
\label{}

Domain theory was introduced by Dana Scott as a mathematical framework for solving the semantic problems in theoretical computer science. Domain theory has connections with several parts of mathematics, including continuous lattice theory, order topology, locale theory and pointless topology, measure theory, and the semantics of nonclassical logics. Within the study of ordered and topological structures, the Scott topology is one of the order topologies induced by the order on a poset and it plays an important role in domain theory. Sober spaces is an important class of $T_{0}$ spaces and it plays an important role in Stone duality(\cite{MS36}), the Hofmann-Mislove Theorem(\cite{GG03}), and representation of closed set lattices(\cite{IWJ81}). A center problem in domain theory is to characterize when a dcpo equipped with the Scott topology is sober. A classical result is that every continuous dcpo or even every quasicontinuous dcpo is sober endowed with the Scott topology(\cite{GG03}). In recent years, there are two distinct strategies for solving this problem. The first approach proceeds constructively, by establishing sufficient conditions under which a $T_{0}$ space is sober. Lawson, Wu and Xi established that every well-filtered core compact $T_{0}$ space is sober(\cite{AT45}). Later, Xu, Sheng, Xi and Zhao proved that every well-filtered first-countable space is sober(\cite{TH62}). More recently, He generalized the latter result by showing that a well-filtered coherent space $X$ is sober provided that $X\times X$ is a Fr\'{e}chet-Urysohn space, or even merely a sequential space(\cite{RBN61,RFRY6}). For the Scott topology on dcpos, Xu and Ji proved that every countable meet-continuous complete lattice endowed with the Scott topology is sober(\cite{EFUP62}). He, Dong and Wang established that every countably core compact dcpo is sober with respect to the Scott topology(\cite{RTU66}).The second approach proceeds by constructing dcpos that satisfy further order conditions but fail to be Scott sober. Miao, Xi, Li and Zhao constructed a countable distributive complete lattice whose Scott topology is not sober(\cite{AB75}). Recently, Sheng, Xi and Zhao showed that the complete Boolean algebra formed by all regular open subsets of the real line fails to be sober in its Scott topology (\cite{EEF32})and He, Dong and Wang proved that the lattice of open sets of the rational number space is not Scott sober(\cite{RTU66}). These results provide strong support for elucidating the intrinsic principles ensuring sobriety of the Scott topology on dcpos.

It follows from the result that every core compact well-filtered (or equivalently, locally compact well-filtered) space is sober(\cite{AT45}). This indicates that a $T_{0}$ space with core compactness, or even with locally compactness is conducive to the sobriety of that space. For a general $T_{0}$ space, core-compactness or local compactness alone does not suffice to guarantee sobriety. For instance, the co-finite topology on an infinite countable set is a locally compact but non-sober space. It is therefore natural to ask whether anything essential changes when the hypothesis is strengthened from that of a locally compact $T_{0}$ space to that of a dcpo whose Scott topology is locally compact. Thus Xu and Zhao raised the following question (Question 3.14 in \cite{U17})):

{\bf Is every locally compact dcpo sober with respect to its Scott topology?}\\
This is the problem that we consider in the present paper. Note that every locally compact space is core-compact. It follows from the result of He, Dong and Wang (Theorem 3.3 in \cite{RTU66}) that every countable locally compact dcpo is sober in its Scott topology. In this paper, we consider Question 3.14 without the restriction of countability. Concretely, by analysing in depth the proof methods of the established results and by using tools from the descriptive set theory, we prove that a locally compact dcpo $P$ is sober with respect to its Scott topology provided that $P$ is a meet-continuous coherent dcpo or that $\Sigma P$ is second-countable.

\section{Preliminaries}
\label{}

\subsection{Order and topology}

\quad Let $P$ be a poset. For each $A\subseteq P$, set $\up A=\{x\in P\mid \exists\ a\in A, a\leq x \}$ and $\dn A=\{x\in P\mid \exists\ a\in A, x\leq a\}$. We write $\dn\{x\}$ as $\dn x$ and $\up\{x\}$ as $\up x$ for each $x\in P$. A subset $D$ of $P$ is {\em directed} provided that $D$ is nonempty and $\forall\ a,b\in D$, $\up a\cap \up b \cap D\neq\emptyset$. Dually, a nonempty subset $D\subseteq P$ is {\em filtered} if $\forall\ a,b\in D$, $\dn a\cap \dn b \cap D\neq\emptyset$. An element $x$ is an {\em upper bound} for a subset $A$ if $a\leq x$ for each $a\in A$. We say that $x$ is a {\em minimal(maximal)} element of $A$ if $x\in A$ and $\dn x\cap A=\{x\}$ ($\up x\cap A=\{x\}$). The notation ${\rm Min(A)}$(${\rm Max(A)}$) denotes the set of all minimal(maximal) elements of $A$. A subset $B\subseteq A$ is {\em cofinal} in $A$ if for every $a\in A$, there is a $b\in B\cap\up a$. For every directed subset $D$ and $d\in D$, $D\cap\up d$ is always cofinal in $D$. A poset $P$ is called a {\em directed complete poset} ({\em dcpo}, for short) if every directed subset of $P$ has a supremum.

\quad Let $P$ be a poset. For each pair of $x, y\in P$, we say that $x$ is {\em way below} $y$, in symbols $x\ll y$, if for each directed set $D\subseteq P$ whenever $\bigvee D$ exists, $y\leq\bigvee D$ implies $\up x\cap D\neq\emptyset$. A poset $P$ is {\em continuous} if $\uuar p=\{b\in P\mid b\ll p\}$ is directed and $\bigvee \uuar p=p$ for every $p\in P$. A subset $U\subseteq P$ is {\em Scott open} if $U=\up U$ and for each directed subset $D$ with $\bigvee D$ existing, $\bigvee D\in U$ implies $D\cap U\neq \emptyset$. The set of Scott open sets of $P$ forms the Scott topology $\sigma(P)$. We write $\Sigma P$ for $(P,\sigma(P))$. The collection of all Scott closed sets of $P$ is simply denoted by $\Gamma P$. For a continuous dcpo $P$, the collection $\{\dda x\mid x\in P\}$ is a basis for the Scott topology $\sigma(P)$. The lower topology $\omega(P)$ on $P$ is the topology generated by $\{P\setminus\up x\mid x\in P\}$. The Lawson topology $\lambda(P)$ is generated by $\sigma(P)\cup\omega(P)$. We write $(P,\lambda(P))$ as $\Lambda P$.

\quad Let $X$ be a topological space. We write \(\operatorname{cl}(A)\) and \(\operatorname{int}(A)\) for the closure and interior of \(A\), respectively. A subset $A$ of $X$ is called a {\em $G_{\delta}$ set} if $A$ is just the intersection of countably many open sets. A subset $F$ of $X$ is {\em irreducible} if $F$ is nonempty and for every pair of closed sets $A, B\subseteq X$, $F\subseteq A\cup B$ implies $F\subseteq A$ or $F\subseteq B$. A $T_{0}$ space $X$ is {\em sober} if every irreducible closed set is just the closure of a singleton. For a $T_{0}$ space $X$, the {\em specialization order} $\leq$ on $X$ is obtained by $x\leq y $ if and only if $x\in cl(\{y\})$. Unless otherwise stated, throughout
the paper, whenever an order-theoretic concept is mentioned in the context of a $T_{0}$ space $X$, it is to be interpreted with respect to the specialization order on $X$. A subset $A$ of a topological space $X$ is {\em saturated} if $A=\up A$ under the specialization order. We use $K(X)$ to denote the set of all compact saturated sets of $X$. A topological space $X$ is called {\em locally compact} if for each open set $U$ and for each $x\in U$, there exist an open set $V$ and a $K\in K(X)$ such that $x\in V\subseteq K\subseteq U$. A topological space $X$ is said to be {\em core compact } if  $\mathcal{O}(X)$ is continuous. We say that a topological space $X$ is {\em coherent} if the intersection of each pair of compact saturated sets is still compact.

\subsection{Descriptive Set Theory}

A topological space $X$ is called a {\em polish space} if $X$ is a separable completely metrizable space. A {\em $\sigma$-algebra} on $X$ is a collection of subsets of $X$ which contains $\emptyset$ and is closed under complements and countably unions. A {\em measurable space} is a pair $(X,\mathbf{S})$, where $\mathbf{S}$ is a $\sigma$-algebra on $X$. A measurable space $(X,\mathbf{S})$ is called a {\em standard Borel space} if there is a polish topology $\tau$ on $X$ such that $\mathbf{S}=\mathbf{B}(\tau)$, where $\mathbf{B}(\tau)$ is the smallest $\sigma$-algebra containing $\tau$. A set $A$ of a polish space $X$ is called analytic if there exists a polish space $Y$ and a continuous mapping $f:Y\longrightarrow X$ with $A=f(Y)$.

\quad Let $\prec$ be a relation on a set $X$. We say the relation $\prec$ is well-founded if there is no sequence $x_{0},x_{1},\cdot\cdot\cdot,x_{n},x_{n+1},\cdot\cdot\cdot$ with $\cdot\cdot\cdot\prec x_{n+1}\prec x_{n}\prec\cdot\cdot\cdot\prec x_{0}\prec x_{1}$. For each $x\in X$, the rank $\rho_{\prec}(x)$ at $x$ is defined by $$\rho_{\prec}(x)={\rm sup}\{\rho_{\prec}(z)+1\mid z\prec x\}\text{ and } {\rm sup }\emptyset=0.$$ We use $\omega_{1}$ to denote the first uncountable ordinal.

\section{Every locally compact coherent meet continuous dcpo is sober}
In this section, we prove that a meet continuous dcpo endowed with the Scott topology is sober if its Scott topology is locally compact and coherent.

\begin{lemma} {\rm (see \cite{GG03}) Let $P$ be a dcpo. Then $P$ is meet continuous if and only if for each $U\in\sigma(P)$ and $x\in U$, $\up(U\cap\dn x)\in\sigma(P)$.}
\end{lemma}

\begin{theorem}{\rm  Let $P$ be a meet-continuous dcpo. If $\Sigma P$ is locally compact and coherent, then $\Sigma P$ is sober.}
\end{theorem}

\begin{proof} We  proof the Theorem 3.2 by establishing the following four claims.

$\mathbf{Claim\ 1:}$ Let $K\subseteq P$ be a compact upper set.
If $$\cdot\cdot\cdot\leq x_{n+1}\leq x_{n}\leq\cdot\cdot\cdot\leq x_3\leq x_2\leq x_1$$
is a decreasing sequence in $K$, then $\bigwedge\limits_{1\leq n}x_{n}\in K$.

Let $L=\bigcap\limits_{1\leq n}\dn x_n$. Thus $L$ is exactly the set of all lower bounds of the sequence.

For each $n\geq1$, $K\cap\dn x_n$ is a nonempty closed subset
in the compact subspace $K$.
By the finite intersection property of $K$, we obtain
$$K\cap L=\bigcap_{1\leq n}(K\cap\dn x_n)\neq\emptyset.$$
In particular, $L$ is nonempty. We next show that $L$ is directed.

Let $a,b\in L$. By coherence of $\Sigma P$, $Q=\up a\cap\up b$
is compact.
For every $1\leq n$, $a,b\leq x_n$ and then $x_n\in Q\cap\dn x_n$.
Thus the sets $Q\cap\dn x_n$ are nonempty closed
subsets of the compact space $Q$, and they form a decreasing
sequence. Similarly,
it follows that
$$
Q\cap L
=
\bigcap_{1\leq n}(Q\cap\dn x_n)
\neq\emptyset.
$$
Choose $c\in Q\cap L$. Then $a,b\leq c$ and $c\in L$. This proves that $L$ is directed.

Since $P$ is a dcpo, the element
$s=\bigvee L$ exists. Furthermore, we can easily deduce that $s=\bigwedge\limits_{1\leq n}x_n\in L$.

Take a $k\in K\cap L$. Then $k\leq s$ and hence $s\in K$ because $K$ is an upper set. This proves Claim 1.

$\mathbf{Claim \ 2:}$ Let $A\subseteq P$ be an irreducible Scott closed set
and  $\{U_n\}_{1\leq n}$ a sequence of Scott open sets
satisfying for each $1\leq n$, $A\cap U_n\neq\emptyset$.
Then $$A\cap(\bigcap\limits_{1\leq n}U_n)\neq\emptyset.$$

For each $1\leq m$, choose $a_{m}\in A\cap U_m$. As $\Sigma P$ is locally compact, there exist a
 Scott open set $V_m$ and a compact upper set $K_m$ such that $$
a_{m}\in A\cap V_m\ \text{and}\
V_m\subseteq K_m\subseteq U_m.
$$

In the following, we construct a triangular array
$c_{j,n}(1\leq j\leq n)$ with the following properties
$$c_{j,n}\in A\cap V_j,\ c_{j,n}\leq c_{j+1,n}\ \text{and}\ c_{j,n+1}\leq c_{j,n}.$$

Firstly, choose any $c_{1,1}\in A\cap V_1$. Let $$W_{1}=\up(V_{1}\cap\dn c_{1,1}).$$ By Lemma 3.1, $W_{1}\in\sigma(P)$. In addition, $c_{11}\in A\cap W_{1}$ and $A\cap V_{2}\neq\emptyset$. The irreducibility of $A$ gives $A\cap W_{1}\cap V_{2}\neq\emptyset$. Choose $c_{2,2}\in A\cap W_{1}\cap V_{2}$. It follows form $c_{2,2}\in W_{1}$ that there is $c_{1,2}\in V_{1}\cap\dn c_{1,1}$ satisfying $c_{1,2}\leq c_{2,2}$. Since $A$ is lower and $c_{1,2}\leq c_{1,1}$, we have $c_{1,2}\in A$ and hence $c_{1,2}\in A\cap V_{1}$. So we have
$$c_{j,2}\in A\cap V_j,j=1,2\ \text{ and }\ c_{1,2}\leq c_{2,2},c_{1,2}\leq c_{1,1}.$$

Suppose that the $n$-th ($2\leq n$) row has been constructed such that
$$c_{1,n}\leq c_{2,n}\leq\cdots\leq c_{n,n}$$ $$\ c_{j,n}\in A\cap V_{j} (1\leq j\leq n)$$ and $$ c_{j,k}\leq c_{j,k-1}(1\leq j\leq k-1,2\leq k\leq n).$$
Let $W_0=P$ and define
$$
W_j
=\up\bigl(
V_j\cap W_{j-1}\cap\dn c_{j,n}
\bigr),\ 1\leq j\leq n.
$$
As $W_{1}$ is open, by introduction and using the Lemma 3.1, each $W_{j}$ is Scott-open.

We assert that for each $1\leq j\leq n$, $c_{j,n}\in W_j$.

For $j=1$, this follows directly from $c_{1,n}\in V_1$
and $W_0=P$.
For $j>1$, suppose that $c_{j-1,n}\in W_{j-1}$.
Since $W_{j-1}$ is an upper set and
$c_{j-1,n}\leq c_{j,n}$, we obtain
$c_{j,n}\in W_{j-1}$.
Together with $c_{j,n}\in V_j$, this yields
$$
c_{j,n}\in
V_j\cap W_{j-1}\cap\downarrow c_{j,n}
\subseteq W_j.
$$
The assertion follows by induction.

In particular, $A\cap W_n\neq\emptyset$.
Since also $A\cap V_{n+1}\neq\emptyset$, irreducibility of
$A$ allows us to choose
$$
c_{n+1,n+1}\in A\cap V_{n+1}\cap W_n.
$$

For every $1\leq i\leq n$,
suppose that $c_{i+1,n+1}\in W_i$ has already been chosen.
By the definition of $W_i$, there exists
$$
c_{i,n+1}
\in
V_i\cap W_{i-1}
\cap\dn c_{i,n}
\cap\dn c_{i+1,n+1}.
$$
This choice ensures that
$$
c_{i,n+1}\leq c_{i,n}
\text{ and }
c_{i,n+1}\leq c_{i+1,n+1}.
$$
Moreover, $c_{i,n}\in A$ and $A$ is a lower set, so
$c_{i,n+1}\in A$.
The membership $c_{i,n+1}\in W_{i-1}$ makes the next step
of the reverse construction possible. So we obtain the elements $$c_{1,n+1},c_{2,n+1},\cdots, c_{n+1,n+1}$$ satisfying all requirements. This completes the induction and constructs the required
triangular array.

\

\begin{tikzpicture}[
    >=stealth,
    point/.style={inner sep=3pt},
    relation/.style={fill=white, inner sep=1pt},
    every node/.style={font=\small}
]
\node[anchor=east] at (-0.8, 0)    {$1:$};
\node[anchor=east] at (-0.8,-1.6)  {$2:$};
\node[anchor=east] at (-0.8,-3.2)  {$3:$};
\node[anchor=east] at (-0.8,-4.8)  {$\vdots$};
\node[anchor=east] at (-0.8,-6.4)  {$n:$};
\node[anchor=east] at (-0.8,-8.0)  {$n+1:$};
\node[point] (a11) at (0,0)    {$c_{1,1}$};
\node[point] (a12) at (0,-1.6) {$c_{1,2}$};
\node[point] (a22) at (2,-1.6) {$c_{2,2}$};
\node[point] (a13) at (0,-3.2) {$c_{1,3}$};
\node[point] (a23) at (2,-3.2) {$c_{2,3}$};
\node[point] (a33) at (4,-3.2) {$c_{3,3}$};
\node[point] (a1n) at (0,-6.4) {$c_{1,n}$};
\node[point] (a2n) at (2,-6.4) {$c_{2,n}$};
\node[point] (a3n) at (4,-6.4) {$c_{3,n}$};
\node[point] (adn) at (6,-6.4) {$\cdots$};
\node[point] (ann) at (8,-6.4) {$c_{n,n}$};
\node[point] (a1p) at (0,-8.0)  {$c_{1,n+1}$};
\node[point] (a2p) at (2,-8.0)  {$c_{2,n+1}$};
\node[point] (a3p) at (4,-8.0)  {$c_{3,n+1}$};
\node[point] (adp) at (6,-8.0)  {$\cdots$};
\node[point] (anp) at (8,-8.0)  {$c_{n,n+1}$};
\node[point] (app) at (10,-8.0) {$c_{n+1,n+1}$};
\foreach \a/\b in {
    a12/a22,
    a13/a23, a23/a33,
    a1n/a2n, a2n/a3n, a3n/adn, adn/ann,
    a1p/a2p, a2p/a3p, a3p/adp, adp/anp, anp/app}{
    \draw[->] (\a) -- node[relation,above] {$\le$} (\b);
}
\foreach \a/\b in {
    a11/a12,
    a12/a13, a22/a23,
    a1n/a1p, a2n/a2p, a3n/a3p, ann/anp}{
    \draw[->] (\a) -- node[relation,right] {$\ge$} (\b);
}
\foreach \a/\b in {a13/a1n, a23/a2n, a33/a3n}{
    \draw[->,dashed] (\a) -- node[relation,right] {$\ge$} (\b);
}
\foreach \x in {0,2,4}{
    \node[fill=white,inner sep=2pt] at (\x,-4.8) {$\vdots$};
}
\node at (6,-4.8) {$\ddots$};
\node[
    anchor=north,
    font=\normalsize,
    yshift=-2mm
] at (current bounding box.south) {\hspace{5cm}The triangular array};
\end{tikzpicture}

Fix $1\leq j$,
the $j$-th column is a decreasing sequence
$$
\cdots\leq c_{j,j+2}\leq c_{j,j+1}\leq c_{j,j}
$$
contained in $V_j\subseteq K_j$.
By Claim 1, the infimum
$d_j=\inf\limits_{j\leq n}c_{j,n}$ belongs to $K_j$.
Since $d_j\leq c_{j,j}\in A$ and $A$ is a lower set, we have $d_j\in A\cap K_j$.

For every $j+1\leq n$, we have $d_j\leq c_{j,n}\leq c_{j+1,n}$.
Hence $d_j$ is a lower bound of the set $\{c_{j+1,n}\mid j+1\leq n\}$.
By the definition of $d_{j+1}$,we get $d_j\leq d_{j+1}$.
Thus $\{d_j\mid 1\leq j\}$ is a chain. As $P$ is a dcpo, $$d=\sup_{1\leq j}d_j$$ exists.
As all $d_j$ belong to $A$, and $A$ is Scott closed, we obtain $d\in A$.
For each $1\leq j$, we have $d_j\in K_j$ and $d_j\leq d$.
As $K_j$ is an upper set, we know $d\in K_j\subseteq U_j$.
Therefore
$$
d\in A\cap\bigcap_{1\leq j}U_j.
$$
Hence,  Claim 2 is proved.

$\mathbf{Claim\ 3:}$ Let $A\subseteq P$ be an irreducible Scott closed set,
and $\{E_n\}_{1\leq n}$ a sequence of Scott open sets such that for each $1\leq n$, $A\cap E_{n}\neq\emptyset$.
There exists
$$
b\in A\cap\bigcap_{1\leq n}E_n
$$
such that $\up b$ is a $G_\delta$ set in $\Sigma P$.

Set $F_0=P$. Choose any $b_{1}\in A\cap E_{1}$. By the locally compactness, there exist a Scott open set $F_{1}$ and a compact upper set $Q_{1}$ such that $$b_{1}\in A\cap F_{1}\ \text{ and }F_{1}\subseteq Q_{1}\subseteq E_{1}\cap F_{0}.$$

Suppose $F_{n}\in\sigma(P)$ and $Q_{n}\in K(P)$ have been chosen such that
$$A\cap F_n\neq\emptyset\ \text{ and }F_n\subseteq Q_n\subseteq E_n\cap F_{n-1}.$$

As $A\cap E_{n+1}\neq\emptyset$ and $A\cap F_{n}\neq\emptyset$, together with the irreducibility of $A$, we have $A\cap E_{n+1}\cap F_{n}\neq\emptyset$. Take any $b_{n+1}\in A\cap E_{n+1}\cap F_{n}$. Using the locally compactness for the case $b_{n+1}\in E_{n+1}\cap F_{n}\in\sigma(P)$, there are $F_{n+1}\in\sigma(P)$ and $Q_{n+1}\in K(P)$ satisfying $$A\cap F_{n+1}\neq\emptyset\ \text{ and }F_{n+1}\subseteq Q_{n+1}\subseteq E_{n+1}\cap F_{n}.$$

By introduction, this produces the required all $F_n$ and $Q_n$ ($1\leq n$).

The construction implies for each $n$,
$$Q_{n+1}\subseteq F_n\subseteq Q_n.$$
In particular, the sequence $\{Q_n\}_{1\leq n}$ is decreasing.

Let $\mathscr C=\{C\in \Gamma P\mid C\subseteq A\ \text{ and for each}\ 1\leq n, C\cap Q_{n}\neq\emptyset\}$.

The family $\mathscr C$ is nonempty because $A\in\mathscr C$.

Let $\{C_i\mid i\in I\}$ be a nonempty chain in $\mathscr C$
under inclusion and put $\widetilde{C}=\bigcap\limits_{i\in I}C_i$.
The set $\widetilde{C}$ is Scott closed and contained in $A$.
Fix $1\leq n$, each $Q_n\cap C_i$
is a nonempty closed subset of the compact space $Q_n$.
As all $C_i$ form a chain, every finite subfamily of $\{Q_n\cap C_i\mid i\in I\}$ has
nonempty intersection. It follows form $Q_{n}\in K(P)$ that
$$
Q_n\cap \widetilde{C}
=
\bigcap_{i\in I}(Q_n\cap C_i)
\neq\emptyset.
$$
In other words, $\widetilde{C}$ belongs to $\mathscr C$. By Zorn's lemma,  $\mathscr C$ contains a minimal element $B$ under the inclusion order.

Next, we prove that $B$ is irreducible.

$B$ is nonempty because it meets the set $Q_1$. Let $O$ be a Scott open set meeting $B$. Then $B\setminus O$ is a proper Scott closed subset of $B$.
The minimality of $B$ implies that $B\setminus O\not\in\mathscr C$.
Clearly, $B\setminus O$ is closed and contained in $A$. As a consequence, there exists a $Q_{n}$ such that
$$
(B\setminus O)\cap Q_n=\emptyset.
$$
Equivalently, $B\cap Q_n\subseteq O$. Now suppose $O$ and $H$ are Scott open sets meeting $B$.
Choose $Q_{r}$ and $Q_{s}$ such that
$$
B\cap Q_r\subseteq O,
\ \text{ and }
B\cap Q_s\subseteq H.
$$
Take $N=\max\{r,s\}$. As the sequence $\{Q_n\}_{1\leq n}$ is decreasing, we obtain
$$
\emptyset\neq B\cap Q_N\subseteq O\cap H.
$$
Thus $O\cap H\cap B\neq\emptyset$, proving that $B$ is irreducible.

For every $n\geq1$, we have $Q_{n+1}\subseteq F_n\ \text{ and }\ B\cap Q_{n+1}\neq\emptyset$.
Hence every $F_n$ meets $B$.
By Claim 2, we can choose 
$$
b\in B\cap\bigcap_{n\geq1}F_n.
$$
In particular, $b\in Q_n\subseteq E_n$ for every $n$.

Since $B$ is a lower set containing $b$, we have
$\dn b\subseteq B$.
The set $\dn b$ is Scott closed and meets every $Q_n$. Thus $\dn b\in\mathscr C$.
Using the minimality of $B$ again, we have $B=\downarrow b$.

For each $1\leq n$, define
$$O_n=\up(F_n\cap\dn b).$$
By Lemma 3.1, each $O_n$ is Scott open.
It follows from $b\in F_n$ that $b\in O_n$.
As $O_n$ is an upper set, it yields that
$\up b\subseteq\bigcap_{n\geq1}O_n$.

Conversely, let $z\in\bigcap\limits_{1\leq n}O_n$.
Then every $O_n$ provides an element $u_n$
such that
$$u_n\in F_n,\ u_n\leq b\ \text{ and }\ u_n\leq z.$$
Consequently, the Scott closed set
$C_{z}=\dn b\cap\dn z$
meets every $Q_n$ because
$$u_n\in C_{z}\cap F_n\subseteq C_{z}\cap Q_n.$$
Moreover, $C_{z}\subseteq\dn b=B\subseteq A.$
Hence $C_{z}\in\mathscr C$ and the minimality of $B$ implies
$C_{z}=B=\dn b$.
In particular, $b\in\dn z$ or equivalently $b\leq z$.
We conclude that $\up b=\bigcap_{n\geq1}O_n.$
Thus $\up b$ is a Scott $G_\delta$ set. The proof of Claim 3 is complete.

$\mathbf{Claim\ 4:}$ $\Sigma P$ is sober.

Let $A\subseteq P$ be an arbitrary irreducible
Scott-closed set and define
$$
D_A=
\{b\in A\mid
  \up b\textnormal{ is a }G_\delta
  \textnormal{ set in}\ \Sigma P\}.
$$
We shall show that $D_A$ is a nonempty directed set and further Scott dense
in $A$.

By Claim 3 and take each $E_n=P$,
this gives an element of $D_A$, proving that $D_A\neq\emptyset$.

Next, let $G$ be any Scott open set meeting $A$.
Using  Claim 3 and taking each $E_n=G$, this gives a point $g\in A\cap G$ such that $\up g$ is
a $G_\delta$ set.
So we have $g\in D_A\cap G$ or alternatively $D_A\cap G\neq\emptyset$.
This proves $\overline{(D_{A})}_{\Sigma P}=A$.

Let $p,q\in D_A$. Then there are Scott open sets
$G_n$ and $H_n$ such that
$$
\up p=\bigcap_{1\leq n}G_n,\ \up q=\bigcap_{1\leq n}H_n.
$$
Clearly, every $G_n$ and $H_n$ meet $A$.
We apply Claim 3 to the interleaved sequence
$$
G_1,H_1,G_2,H_2,\cdot\cdot\cdot,G_n,H_n\cdot\cdot\cdot
$$
and hence there exists a $l\in A$ such that $\up l$ is a
Scott $G_\delta$ set and
$$
l\in
\left(\bigcap_{1\leq n}G_n\right)
\cap
\left(\bigcap_{1\leq n}H_n\right)
=
\up p\cap\up q.
$$
Equivalently, $l\in D_A$ and $ p, q\leq l$. Thus, $D_A$ is directed.

Since $P$ is a dcpo, the supremum
$a^{\ast}=\sup D_A$ exists.
The set $A$ is Scott closed and contains $D_A$. So $A$ contains $a^{\ast}$ and further
$\dn a^{\ast}\subseteq A$ as $A$ is lower.
Besides, $\dn a^{\ast}$ is a Scott closed set containing
$D_A$. Hence, $$A=\overline{(D_{A})}_{\Sigma P}\subseteq\dn a^{\ast}.$$ We have proved $A=\dn a^{\ast}=\overline{\{a^{\ast}\}}_{\Sigma P}$. Therefore $\Sigma P$ is sober.
\end{proof}

\section{Every second countable locally compact dcpo is sober}
In this section, we prove that for a dcpo $P$, if $\Sigma P$ is locally compact and second countable, then $\Sigma P$ is sober.

\begin{lemma} {\rm  Let $P$ be a dcpo and $V, U\in\sigma(P)$. If $V\ll U$ and $\{a_{n}\}_{n\in\mathbb{N}}$ is a sequence contained in $V$, then there exists $b\in U$ such that
$\bigl|\{n\in\mathbb N\mid b\leq a_n\}\bigr|=\infty$.
}
\end{lemma}

\begin{proof}
Assume that such $b$ does not exist. For each $m\in\mathbb{N}$, put
$$
C_m=(P\setminus U)\cup\bigcup_{m\leq n}\downarrow a_n.
$$
Each $C_{m}$ is clearly a lower set. Let $D\subseteq C_m$ be a directed set and suppose $d=\bigvee D$. It is no doubt that $\bigvee D\in C_{m}$ whenever $\bigvee D\in P\setminus U$. It remains to consider the condition $d\in U$. As $U\in\sigma(P)$,
there is $d_0\in D\cap U$. This leads to $\up d_{0}\subseteq U$. Take
$$D_0=D\cap\up d_0\ \text{ and }I_m=\{m\leq n\mid d_0\leq a_n\}.$$
As $D_{0}\subseteq C_{m}\cap \up d_{0}\subseteq C_{m}\cap U$, we deduce that $D_{0}\subseteq\dn I_{m}$.
By assumption, $I_m$ is finite. Consequently, $\dn I_{m}$ is a Scott closed subset and hence
$$\bigvee D_0=\bigvee D=d\in\dn I_{m}\subseteq C_{m}.$$
Thus each $C_{m}$ is Scott closed.
By assumption, $\bigcap\limits_{m\in\mathbb{N}} C_m=P\setminus U$. In other words, $$U=\bigcup\limits_{m\in\mathbb{N}}(P\setminus C_m)=\bigvee\limits_{m\in\mathbb{N}}\limits^{\sigma(P)}(P\setminus C_{m}).$$ Since $V\ll U$, there exists $C_{m_{0}}$ such that $V\subseteq P\setminus C_{m_{0}}$. This contradicts with $a_{m_{0}}\in V\cap C_{m_{0}}$. The proof is complete.
\end{proof}

\begin{lemma} {\rm  Let $P$ be a dcpo such that $\Sigma P$ is core-compact. If $\{A_{n}\mid 1\leq n\}$ is a decreasing sequence consisting of lower sets, then $cl_{\Sigma P}(\bigcap\limits_{1\leq n}A_{n})=\bigcap\limits_{1\leq n}cl_{\Sigma P}(A_{n})$.
}
\end{lemma}

\begin{proof} Clearly, $cl_{\Sigma P}(\bigcap\limits_{1\leq n}A_{n})\subseteq\bigcap\limits_{1\leq n}cl_{\Sigma P}(A_{n})$. Conversely, choose any $x\in\bigcap\limits_{1\leq n}cl_{\Sigma P}(A_{n})$ and $U\in\sigma(P)$ with $x\in U$. As $\Sigma P$ is core-compact, there exists a Scott open set $V$ satisfying $x\in V$ and $V\ll U$. It follows from $x\in\bigcap\limits_{1\leq n}cl_{\Sigma P}(A_{n})$ that $V\cap A_{n}\neq\emptyset$ for each $1\leq n$. Fix a $a_{n}\in V\cap A_{n}$. By Lemma 4.1, there is $b\in U$ such that $\bigl|\{n\in\mathbb N\mid b\leq a_n\}\bigr|=\infty$. Consequently, for each $n$, there exists $n\leq k$ such that $b\leq a_{k}$. Whence, $a_{k}\in A_{k}\subseteq A_{n}$ and for each $n$, $b\in\dn A_{n}\cap U=A_{n}\cap U$. So we conclude that $b\in U\cap\bigcap\limits_{1\leq n}A_{n}$ and thus $x\in cl_{\Sigma P}(\bigcap\limits_{1\leq n}A_{n})$. Therefore, $cl_{\Sigma P}(\bigcap\limits_{1\leq n}A_{n})=\bigcap\limits_{1\leq n}cl_{\Sigma P}(A_{n})$.\end{proof}

\begin{lemma} {\rm (see \cite{R2226}) Let $X$ be a $T_{0}$ space and $K\in K(X)$. Then $K=\up{\rm Min}(K)$ and ${\rm Min}(K)$ is compact in $X$.
}
\end{lemma}

\begin{lemma} {\rm (\cite{GG03})  Let $P$ be a continuous lattice. Then $\Lambda P$ is a compact Hausdorff space.
}
\end{lemma}

\begin{lemma} {\rm (see \cite{RE77})A compact Hausdorff space is metrizable if and only if it is second countable.}
\end{lemma}

\begin{lemma} {\rm  Let $P$ be a dcpo. If $\Sigma P$ is core compact and second countable, then the Lawson topology $\lambda(\sigma(P))$ is second countable and compact Hausdorff. In particular, $\Lambda(\sigma(P))$ is metrizable.
}
\end{lemma}

\begin{proof} As $\Sigma P$ is core compact, $\sigma(P)$ is a continuous lattice. By Lemma 4.4, $\lambda(\sigma(P))$ is compact and Hausdorff. Since $\Sigma P$ is second countable, we can enumerate $\mathscr{B}=\{U_{n}\mid n\in\mathbb{N}\}$ as a topology base of $\Sigma P$. For each finite subset $F\subseteq\mathbb{N}$, set $V_{F}=\bigcup\limits_{k\in F}U_{k}$. Set $$\mathscr{D}=\{\dda V_{F}\mid F \text{ is a finite subset of }\mathbb{N}\}.$$

$\mathbf{Claim \ 1:}$ $\mathscr{D}$ is a basis for $\Sigma\sigma(P)$.

 Let $\mathscr{U}\in\sigma(\sigma(P))$ and $U\in\mathscr{U}$. As $\sigma(P)$ is continuous, $\{\dda W\mid W\in\sigma(P)\}$ is a basis for $\Sigma\sigma(P)$. So there exists a $\widetilde{W}\in\sigma(P)$ such that $U\in\dda\widetilde{W}\subseteq\mathscr{U}$. By the interpolation property of way below relation, we can choose $\widetilde{U},\widehat{W}\in\sigma(P)$ satisfying $\widetilde{W}\ll\widehat{W}\ll\widetilde{U}\ll U$.
 Since $\mathscr{B}$ forms a base of $\Sigma P$, we can fix a subset $A\subseteq\mathbb{N}$ such that $\widetilde{U}=\bigcup\limits_{m\in A} U_{m}$. Take $$\mathscr{C}=\{V_{F}\mid F\text{ is a finite subset of }A\}.$$ Then $\mathscr{C}$ is directed and $\widetilde{U}=\bigcup\mathscr{C}$. It follows from $\widehat{W}\ll\widetilde{U}$ that $\widehat{W}\subseteq V_{F_{0}}\subseteq\widetilde{U}$ for some finite subset $F_{0}\subseteq A$. As a consequence, $\widetilde{W}\ll V_{F_{0}}\ll U$, equivalently, $$U\in\dda V_{F_{0}}\subseteq\dda\widetilde{W}\subseteq\mathscr{U}.$$ This proves Claim 1.

Set $$\mathscr{E}=\{\sigma(P)\setminus\up V_{F}\mid F\text{ is a finite subset of }\mathbb{N}\}.$$

$\mathbf{Claim \ 2:}$ $\mathscr{E}$ is a topology base for the lower topology $\omega(\sigma(P))$.

For each $H\in\sigma(P)$, we assert $\sigma(P)\setminus\up H=\bigcup\limits_{V_{F}\ll H}(\sigma(P)\setminus\up V_{F})$. Equivalently, $\up H=\bigcap\limits_{V_{F}\ll H}\up V_{F}$.
Clearly, $\up H\subseteq\bigcap\limits_{V_{F}\ll H}\up V_{F}$. Let $O\in\bigcap\limits_{V_{F}\ll H}\up V_{F}$. In other words, $O\in\sigma(P)$ and $\bigcup\limits_{V_{F}\ll H}V_{F}\subseteq O$. As $\sigma(P)$ is continuous, $H=\bigcup\limits_{K\ll H}K$. Choose any $K\ll H$. Repeat the process in Claim 1, we can find a $V_{F_{1}}$ such that $K\ll\ll V_{F_{1}}\ll H$. This means that $$\bigcup\limits_{V_{F}\ll H}V_{F}=H\subseteq O.$$ Consequently, $\bigcap\limits_{V_{F}\ll H}\up V_{F}\subseteq\up H$. This proves Claim 2.

By Claim 1 and Claim 2, $\lambda(\sigma(P))=\sigma(\sigma(P))\vee\omega(\sigma(P))$ is second countable. Hence, by Lemma 4.5, $\lambda(\sigma(P))$ is metrizable. \end{proof}

\begin{lemma} {\rm (\cite{RE77})  Let $(X,d)$ be a metric space. Then for $r\in\mathbb{R}$, $$\{(x,y)\in X^{2}\mid r\leq d(x,y)\}$$ is closed in $X\times X$.}
\end{lemma}

\begin{lemma} {\rm  Let $P$ be a continuous dcpo. Then $\{(x,y)\in L^{2}\mid x\leq y\}$ and $\{(x,y)\in L^{2}\mid y\leq x\}$ are closed in $\Lambda L\times\Lambda L$.}
\end{lemma}

\begin{proof} Let $A=\{(x,y)\in L^{2}\mid x\not\leq y\}$ and choose any $(a,b)\in A$. As $P$ is continuous, there exists $z\ll a$ such that $z\not\leq b$. Set $$U=\dda z\text{ and }V=P\setminus\up z.$$ Then we have $(a,b)\in U\times V$. Suppose $(u,v)\in U\times V$. Equivalently, $z\not\leq v$ and $z\ll u$. This leads to $z\leq u$. Assume $u\leq v$. It is immediately that $z\leq v$, impossible. It follows that $(a,b)\in U\times V\subseteq A$. As a consequence, $A$ is open in $\Lambda L\times\Lambda L$. Similarly, $\{(x,y)\in L^{2}\mid y\not\leq x\}$ is still open in $\Lambda L\times\Lambda L$. Therefore, their complements are closed.
\end{proof}

\begin{lemma} {\rm (\cite{RE77})  Let $(X,d)$ be a compact metric space. Then $X$ is a polish space.}
\end{lemma}

\begin{lemma} {\rm (\cite{EWWF31})  Let $X$ be a polish space and $G$ a $G_{\delta}$ set. Then the subspace $G$ is a polish space.}
\end{lemma}

\begin{lemma} {\rm (\cite{EWWF31})  Let $\{X_{n}\mid n\in\mathbb{N}\}$ be a family of polish spaces. Then the countable topological sum $\coprod\limits_{n\in\mathbb{N}}X_{n}$ is a polish space.}
\end{lemma}

\begin{lemma} {\rm (\cite{EWWF31})  Let $X$ be a standard Borel space and $R$ an analytic well-founded relation. Then there is a $\gamma<\omega_{1}$ such that for each $x\in X$, $\rho_{R}(x)<\gamma$.}
\end{lemma}

\begin{theorem}
{\rm Let $P$ be a dcpo such that $\Sigma P$ is locally compact and second countable. Then $\Sigma P$ is sober.}
\end{theorem}

\begin{proof}

Let $F$ be an irreducible closed subset in $\Sigma P$.

$\mathbf{Claim\ 1:}$ $\sigma(F)=\{U\cap F\mid U\in\sigma(P)\}$.

It is not difficult to check that each $U\cap F$ is Scott open in the subdcpo $F$. Suppose $\widetilde{U}$ is open in $\Sigma F$. Set $\widehat{U}=P\setminus(F\setminus\widetilde{U})$. It follows from $\widetilde{U}\in\sigma(F)$ and $F\in\Gamma P$ that $\widehat{U}\in\sigma(P)$. Besides, $$\widehat{U}\cap F=(P\setminus(F\setminus\widetilde{U}))\cap F=((P\setminus F)\cup\widetilde{U})\cap F=\widetilde{U}.$$ Hence, Claim 1 is proved.

Since $\Sigma P$ is locally compact and second countable, using Claim 1, $\Sigma F$ is also locally compact and second countable.

From now on, all orders, open sets and closures are taken in this
inherited dcpo $F$.

We enumerate a countable base of $\Sigma F$ as $$\mathscr B_0=\{B_n\mid n\in\mathbb N, B_{n} \mbox{ is a nonempty open set in } \Sigma F\}$$ and put
$U_n=\bigcap\limits_{i=0}\limits^{n}B_i$. As $F$ is irreducible, each $U_{n}$ is still nonempty. In addition, for each nonempty $O\in\sigma(F)$, there exist a $j\in\mathbb{N}$ such that $U_{j}\subseteq O$.

In the following, assume that $\bigcap\limits_{n\in\mathbb{N}}U_{n}=\emptyset$.

As $\Sigma F$ is locally compact, we can
choose compact saturated subsets $K_{n}\in K(F)$ such that
$$
\varnothing\neq\operatorname{int}(K_n)
\subseteq K_n\subseteq U_n.
$$

By Lemma 4.3, For every nonempty compact subset $K_{n}$,
$K_{n}=\uparrow\operatorname{Min}K_{n}$.

Put $$S=\bigcup_{j\in\mathbb N}\operatorname{Min}K_j.$$ As each nonempty subset $\operatorname{Min}K_n$ is contained in $U_{n}$. We assert that $cl(S)_{\Sigma F}=F$.

For all countable ordinals $\lambda$, define
$$S_0=S,\ S_{\alpha+1}=S_\alpha\cap
  \bigcap_{n\in\mathbb N}\dn(S_\alpha\cap U_n),$$ and if $0\leq\lambda<\omega_{1}$ is a limit ordinal,
$$S_\lambda=\bigcap_{\alpha<\lambda}S_\alpha.$$

This induction gives that $S_\alpha\subseteq S_\beta$ if $\beta\leq\alpha<\omega_1$.

$\mathbf{Claim\ 2:}$ If $s\in S_\alpha,r\in S$ and $ r\leq s$, then $r\in S_\alpha$.

The statement holds when $\alpha=0$.

Suppose this result holds for every countable ordinal $\alpha$. Let $s\in S_{\alpha+1}$ and $r\leq s$ for some $r\in S$. Then $s\in S_{\alpha}$ and for each $n\in\mathbb{N}$, there exists $s_{n}\in S_{\alpha}\cap U_{n}$ satisfying $s\leq s_{n}$. So $r\leq s_{n}$ and further $r\in\bigcap_{n\in\mathbb N}\downarrow(S_\alpha\cap U_n)$.
 By assumption, $r\in S_{\alpha}$. Thus, $r\in S_{\alpha+1}$.

Suppose $\alpha< \omega_{1}$ is a limit ordinal and the statement is valid for all $\beta<\lambda$. The construction of $S_{\alpha}=\bigcap\limits_{\beta<\alpha}S_\beta$ shows that $S_{\alpha}$ is a lower set in $S$.

Therefore, Claim 2 is proved.

$\mathbf{Claim\ 3:}$ $\forall\ \alpha<\omega_1$, $cl(S_\alpha)_{\Sigma F}=F$.

We proceed by transfinite induction on $\alpha$.
The case $\alpha=0$ is immediate by the fact $cl(S)_{\Sigma F}=F$.

Suppose first that $\alpha<\omega_{1}$ is not a limit ordinal and $cl(S_\alpha)_{\Sigma F}=F$.
Set $$A_n=\dn(S_\alpha\cap U_n).$$ As $F$ is irreducible in $\Sigma F$, we deduce that $cl(A_n)_{\Sigma F}=F$. By Lemma 4.2, $cl(\bigcap\limits_n A_n)_{\Sigma F}=F$. Let $\widetilde{O}\in\sigma(F)$ be a nonempty open set. We can fix a $U_{j}\subseteq \widetilde{O}$. Choose a $$z\in \operatorname{int}K_j\cap\bigcap\limits_n A_n.$$ Then there exist a $m\in{\rm Min}(K_{j})$ satisfying $m\leq z$. At the same time, for each $n$, we can fix a $t_{n}\in S_{\alpha}\cap U_{n}$ such that $z\leq t_{n}$.
By Claim 2, $m\in S_{\alpha}$ and further $m\in\bigcap\limits_{n}A_{n}$. So we obtain $$m\in S_{\alpha+1}\cap{\rm Min}(K_{j})\subseteq S_{\alpha+1}\cap K_{j}\subseteq S_{\alpha+1}\cap U_{j}\subseteq S_{\alpha+1}\cap\widetilde{O}.$$ Now, we assert that $cl(S_{\alpha+1})_{\Sigma F}=F$. This proves the successor step.

Finally, suppose that
$\alpha$ is a nonzero limit ordinal. Suppose that $cl(S_\beta)_{\Sigma F}=F$ for every
$\beta<\alpha$. Choose an increasing cofinal
sequence $\{\beta_{n}\mid n\in\mathbb{N}\}$ such that
$$\beta_n<\alpha \text{ and } \sup\limits_n\beta_n=\alpha.$$ Then we get $S_\alpha=\bigcap\limits_n S_{\beta_n}$. By assumption and Lemma 4.2,

$$
cl(\bigcap_n\dn S_{\beta_n})_{\Sigma F}
=
\bigcap_n cl(\dn S_{\beta_n})_{\Sigma F}
=F.
$$
For any nonempty open set $\widehat{O}$, choose
$$
U_i\subseteq \widehat{O}
\text{ and }z_{1}\in\operatorname{int}K_i \cap\bigcap_n\dn S_{\beta_n}.$$  Select a $m_{1}\in\operatorname{Min}K_i$ such that $m_{1}\leq z_{1}$.

For every index $n$, there is
$r_n\in S_{\beta_n}$ satisfying $z_{1}\leq r_n$. By Claim 2 again, $m_{1}\in\bigcap\limits_{n} S_{\beta_n}$. It follows from $$\operatorname{Min}K_i\subseteq K_{i}\subseteq U_{i}\subseteq\widehat{O}$$ that $m_{1}\in S_{\alpha}\cap \widehat{O}$. Consequently, $cl(S_{\alpha})_{\Sigma F}=F$. We finish the proof of Claim 3.

Furthermore, by Claim 3 and the irreducibility of $F$, it follows that for each $n$ and $\alpha$,
$$cl(S_\alpha\cap U_n)_{\Sigma F}=F.$$

By Lemma 4.6, $\Lambda(\sigma(P))$ is a compact metrizable space. Define a mapping $c:\Gamma(P)\longrightarrow\sigma(P)$ by $c(A)=P\setminus A$. The topology $\tau$ on $\Gamma(P)$ is given by
$$
\tau=c^{-1}(\lambda(\sigma(P))).
$$
Then $M=(\Gamma(P),\tau)$ is also a compact metrizable space because $c$ is a homeomorphism between $M$ and $\Lambda(\sigma(P))$.

Define the order embedding $\eta:F\longrightarrow M$ by $\eta(x)=\dn x$. For each open set $U$, let
$$
\Diamond U=\{A\in\Gamma(P)\mid A\cap U\neq\varnothing\}.
$$
Then we have $$c(\Diamond U)=\{B\in\sigma(P)\mid (P\setminus B)\cap U\neq\emptyset\}=\sigma(P)\setminus\up U\in\lambda(\sigma(P)).$$ In other words, $\Diamond U\in\tau$. In addition, $\eta^{-1}(\Diamond U)=U$ as $U$ is upper in $P$.

For each $K\in K(P)$, put
$$
Q_K=\{A\in M\mid A\cap K\neq\varnothing\}.
$$
It follows from the compactness of $K$ that $$\Phi(K)
=\{U\in \sigma(P)\mid K\subseteq U\}$$ is Scott open in $\sigma(P)$. One can easily verify that
$c(M\setminus Q_K)=\Phi(K)$. Consequently, $Q_K$ is closed in $M$. Hence $Q_K$ is compact.

$\mathbf{Claim\ 4:}$ $\operatorname{Min}Q_K=\eta(\operatorname{Min}K)$.

Let $A\in\operatorname{Min}Q_K$. Choose a $k\in A\cap K$. Then $\dn k\in Q_{K}$ and $\dn k\subseteq A$. The minimality of $A$ implies $A=\dn k=\eta(k)$. Suppose there is a $k_{1}\in K$ such that $k_{1}\leq k$. Then $\dn k_{1}\in Q_{K}$ and $\dn k_{1}\subseteq\dn k$. Using $A\in\operatorname{Min}Q_K$ again, $\dn k_{1}=\dn k$. Equivalently, $k\in\operatorname{Min}K$.

Conversely, choose any $k\in\operatorname{Min}K$. Suppose $B\in Q_K$ and $B\subseteq\dn k$. Then we can fix a $b\in B\cap K$ such that $b\leq k$. It follows from $k\in\operatorname{Min}K$ that $b=k$. As $B$ is lower in $P$, $\dn b=\dn k\subseteq B$. As a result, $B=\dn k$, proving that $\eta(k)\in\operatorname{Min}Q_K$.

Therefore, Claim 4 is obtained.

Fix a metric $d:M\times M\longrightarrow[0,\infty)$. For every positive integer $r$, define
$$
H_r(K)=
\left\{
(A_{1},A_{2})\in Q_K\times Q_K\mid
A_{1}\subseteq A_{2}\text{ and }\frac{1}{r}\leq d(A_{1},A_{2})
\right\}.
$$
Actually, $H_r(K)=(Q_K\times Q_K)\cap\mathscr{A}_{1}\cap\mathscr{A}_{2}$, where $$\mathscr{A}_{1}=\{(A_{1},A_{2})\in M^{2}\mid A_{1}\subseteq A_{2}\}\text{ and } \mathscr{A}_{2}=\{(A_{1},A_{2})\in M^{2}\mid \frac{1}{r}\leq d(A_{1},A_{2})\}.$$ By Lemma 4.7 and Lemma 4.8, $\mathscr{A}_{1}$ and $\mathscr{A}_{2}$ are closed in $M^{2}$. As $Q_{K}$ is compact in $M$, each $H_r(K)$ is compact in $M^{2}$.

$\mathbf{Claim\ 5:}$ $\operatorname{Min}Q_K$ is a $G_{\delta}$ subset in $M$.

As $H_r(K)$ is compact in $M^{2}$, $\pi_{2}(H_r(K))$ is compact in $M$. Then $\pi_{2}(H_r(K))$ and $Q_{K}$ are closed in $M$ because $M$ is a compact Hausdorff space. Furthermore, we assert that
$$
Q_K\setminus\operatorname{Min}Q_K
=
\bigcup_{r\geq1}\pi_2(H_r(K)).$$

Take any $\widehat{A}\in Q_K\setminus\operatorname{Min}Q_K$. Then $\widehat{A}\cap K\neq\emptyset$ and $\widehat{A}\in\Gamma(P)$. Choose $\hat{a}\in \widehat{A}\cap K$. By Lemma 4.3, there exists $\hat{k}\in K$ such that $\hat{k}\leq\hat{a}$. It follows from $\widehat{A}\not\in\operatorname{Min}Q_K=\eta(\operatorname{Min}K)$ that $\dn\hat{k}$ is a proper subset of $\widehat{A}$. This means that $\frac{1}{\hat{r}}\leq d(\dn\hat{k},\widehat{A})$ for some positive integer $\hat{r}$. Consequently, $\widehat{A}\in\pi_2(H_{\hat{r}}(K))$.

Conversely, let $\widetilde{A}\in\pi_2(H_{\tilde{r}}(K))$ ($1\leq\tilde{r}$). Fix a $\widetilde{B}\in Q_{K}$, $\widetilde{B}\subseteq \widetilde{A}$ and $\frac{1}{\tilde{r}}\leq d(\widetilde{B},\widetilde{A})$. So we have $\widetilde{B}\neq\widetilde{A}$. Hence $\widetilde{A}$ is not a minimal element in $Q_{K}$.

Now, we have proved $Q_K\setminus\operatorname{Min}Q_K=\bigcup\limits_{r\geq1}\pi_2(H_r(K))$. As a consequence, $$\operatorname{Min}Q_K=Q_K\setminus\bigcup\limits_{r\geq1}\pi_2(H_r(K))=Q_K\cap(\bigcap\limits_{1\leq r}(P\setminus\pi_2(H_r(K)))).$$
By Lemma 4.9, the closed set $Q_{K}$ is a $G_{\delta}$ set. Hence, $\operatorname{Min}Q_K$ is also a $G_{\delta}$ set in $M$. The Claim 5 is proved.

For each $n,j\in\mathbb{N}$, set
$$E_{n,j}=\operatorname{Min}Q_{K_j}\cap\Diamond U_n.$$

By Lemma 4.9, Lemma 4.10 and Claim 5, each subspace $(E_{n,j},\tau|_{E_{n,j}})$ is Polish. Consider the topological sum
$$
Z=\coprod_{(n,j)\in\mathbb N^2}E_{n,j}.$$
By Lemma 4.11, $Z$ is polish. Note that $A\in E_{n,j}$ if and only if $A=\dn a$ for some $a\in\operatorname{Min}K_j\cap U_n$. For convenience, we can write each element of $Z$ as
$$(n,j,\dn s),\text{ for some } s\in\operatorname{Min}K_j\cap U_n.$$ Define a mapping
$\pi:Z\longrightarrow M$ by $\pi(n,j,\dn s)=\eta(s)$. Clearly, the map $\pi$ is continuous by the definition of the topological sum.

Define a relation $\prec$ on $Z$ by
$$
(n+1,k,\dn t)\prec(n,j,\dn s)
\quad\Longleftrightarrow\quad
\dn s\subseteq\dn t.$$

$\mathbf{Claim\ 6:}$ There exists a $\gamma<\omega_{1}$ such that
$\forall(n,j,\dn s)\in Z,\ \rho_{\prec}(n,j,\dn s)<\gamma$.

We assert that the relation $\prec$ is a analytic well-founded relation.

Suppose $(x_{m},y_{m})_{m\in\mathbb{N}}\subseteq\prec$ is a sequence converging to $(x,y)$. For convenience, let $$x_{m}=(n_{m}+1,k_{m},\dn t_{m})\text{ and }y_{m}=(n_{m},j_{m},\dn s_{m})$$ with $\dn s_{m}\subseteq\dn t_{m}$. As different parts $E_{n,j}$ are disjoint and clopen, we can fix $x\in E_{p,q}$,$y\in E_{l,r}$ and $N\in\mathbb{N}$ such that for each $N\leq m$, $(x_{m},y_{m})\in E_{p,q}\times E_{l,r}$. Then we obtain
$$p=n_{m}+1,q=k_{m},l=n_{m},r=j_{m}(N\leq m).$$
Now $\{(p,q,\dn t_{m})\}_{N\leq m}$ and $\{(l,r,\dn s_{m})\}_{N\leq m}$ are two sequences converging to $x$ and $y$ in $E_{p,q}$ and $E_{l,r}$ respectively. For convenience, we fix $x=(p,q,\dn t)$ and $y=(l,r,\dn s)$. As a consequence, $\{\dn t_{m}\}_{N\leq m}$ and $\{\dn s_{m}\}_{N\leq m}$ converge to $\dn t$ and $\dn s$ in $M$ respectively. Then the sequence $\{(\dn s_{m},\dn t_{m})\}_{N\leq m}$ contained in $\mathscr{H}$ converges to $(\dn s,\dn t)$, where $\mathscr{H}=\{(A,B)\in M^{2}\mid A\subseteq B\}$. By Lemma 4.8, $$\mathscr{H}=(c\times c)^{-1}(\{(U,V)\in\sigma(F)\times\sigma(F)\mid V\subseteq U\})$$ is closed in $M^{2}$. Then we have $(\dn s,\dn t)\in\mathscr{H}$ since $M^{2}$ is metrizable. In other words, $\dn s\subseteq\dn t$ and hence $x\prec y$. This proves that $\prec\subseteq Z^{2}$ is sequentially closed and further closed. Note that every closed set is a analytic set. So we deduce that $\prec$ is analytic.

Assume that the relation $\prec$ fails to be well-founded. Then there is a decreasing chain $$(n_{0},j_{0},\dn s_{0})\succ(n_{1},j_{1},\dn s_{1})\succ\cdot\cdot\cdot\succ(n_{i},j_{i},\dn s_{i})\succ(n_{i+1},j_{i+1},\dn s_{i+1})\succ\cdot\cdot\cdot.$$ By the definition of $\prec$, we have
$$
s_0\leq s_1\leq\cdots\leq s_{i}\leq s_{i+1}\leq\cdots
\text{ and }
s_i\in U_{n_0+i}.
$$
As $F$ is Scott closed in $P$, $s=\sup\limits_i\limits^{P} s_i\in F$. For each $n$, choose a $i$ such that $n\leq n_{0}+i$. Then we have
$$s_i\in U_{n_0+i}\subseteq U_n$$ and further $s\in U_n$ because $U_{n}$ is upper. This shows that $s\in\bigcap\limits_{n\in\mathbb{N}}U_{n}$, a contradiction. Hence $\prec$ is well-founded.

We determine the measure $\mathbf{S}$ on $Z$ by taking $\mathbf{S}=\mathbf{B}(\tau_{Z})$, where $\tau_{Z}$ is the topology on $Z$. Now $(Z,\mathbf{S})$ is a measure space and $\prec$ is analytic on $Z\times Z$ and well-founded. Using Lemma 4.12, Claim 6 is valid.

$\mathbf{Claim\ 7:}$ $\forall\alpha<\omega_1,\forall(n,j,\dn s)\in Z$, if $ s\in S_\alpha$, then $\rho_{\prec}(n,j,\dn s)\geq\alpha$.

The initial case $\alpha=0$ is immediate.

Suppose the condition holds when $0<\alpha<\omega_{1}$ is not a limit ordinal. $\forall(n,j,\dn s)\in Z$ with $s\in S_{\alpha+1}$,
$$\exists t\in S_\alpha\cap U_{n+1},\ s\leq t.
$$
Since $S_{\alpha}\subseteq S_{0}=\bigcup\limits_{j\in\mathbb{N}}{\rm Min}K_{j}$, we can choose $k$ such that $t\in\operatorname{Min}K_k$.
Then $(n+1,k,\dn t)\prec(n,j,\dn s)$ and hence $\rho_{\prec}(n,j,\dn s)\geq\rho_{\prec}(n+1,k,\dn t)+1\geq\alpha+1.$ This shows the successor step.

Suppose $\alpha$ is not a nonzero limit ordinal and the statement hold for all $\beta<\alpha$. Choose any $(n,j,\dn s)\in Z$ and $s\in S_{\alpha}$. Then $s\in\bigcap\limits_{\beta<\alpha}S_{\beta}$ and by assumption
$\rho_{\prec}(n,j,\dn s)\geq\beta$. Hence we obtain $$\alpha=\bigvee\limits_{\beta<\alpha}\beta\leq\rho_{\prec}(n,j,\dn s).$$

Therefore, Claim 7 holds for all countable ordinals.

By Claim 3, we take $\gamma$ the ordinal number presented in Claim 6 and get $cl(S_\gamma)_{\Sigma F}=F$. As $U_0\neq\varnothing$, there exists $s\in S_\gamma\cap U_0$. Choose a label $j$ such that $s\in\operatorname{Min}K_j$. This label $j$ must exist as $S_\gamma\subseteq S_{0}$. By Claim 7 and Claim 6, $\gamma\leq\rho(0,j,\dn s)<\gamma$, a contradiction.

Therefore, the assumption $\bigcap\limits_{n\in \mathbb{N}}=\emptyset$ is impossible.

Now we have $\bigcap\limits_{n\in\mathbb{N}} U_{n}\neq\emptyset$ and choose any  $p\in\bigcap\limits_{n\in\mathbb{N}} U_{n}$. Then $p$ belongs to each nonempty open set in $\Sigma F$. We deduce that $p$ is the largest element in $F$. Consequently, $$F=\dn_{F} p=\dn_{P}p=cl(\{p\})_{\Sigma P}.$$ Hence $\Sigma P$ is sober.
\end{proof}


\section*{References}
\begin{thebibliography}{s9}

\bibitem{IWJ81} D. A. Drake, W. J. Thron, On the representation of of an abstract lattice as the family of closed sets of a topological space, Transactions of the American Mathematical Society, 120 (1965), 57-71.

\bibitem{R2226} M. Ern\'{e}, Infinite distributive  laws  versus  local  connectedness  and  compactness  properties, Topology and its Applications, 156 (2009) 2054-2069.

\bibitem{RE77} R. Engelking, General Topology, Polish Scientific Publishers, Warszawa, 1977.



\bibitem{GG03} G. Gierz, K. H. Hofmann, K. Keimel, J. D. Lawson, M. W. Mislove and D. S. Scott, Continuous Lattices and Domains, Encyclopedia of Mathematics and its Applications, vol. 93, Cambridge University Press, Cambridge, 2003.


\bibitem{RBN61} Z. He, The sobriety of Fr\'{e}chet-Urysohn  spaces and Scott topology of the lattice of open sets, Houston Journal of Mathematics, 52(1)(2026) 195-212.



\bibitem{RFRY6} Z. He, Applications of sequential spaces to Smyth power spaces and sober spaces, 	arXiv:2609.27775, 2026.

 \bibitem{RTU66} Z. He, Z. Dong, K. Wang, Sobriety of Scott topologies under countability conditions, arXiv:2609.24001, 2026.


\bibitem{EWWF31} A. S. Kechris, Classical Descreptive Set Theory, Graduate Texts in Mathematics vol. 156, Springer, New York, 1995.

\bibitem{AT45} J. D. Lawson, G. Wu, X. Xi, Well-filtered spaces, compactness, and the lower topology, Houston Journal of Mathematics, 46(1) (2020) 283-294.

\bibitem{AB75} H. Miao, X. Xi, Q. Li and D. Zhao, Not every countable complete distributive lattice is sober, Mathematical Structures in Computer Science, 33 (2023) 809-831.


\bibitem{MS36} M. H. Stone, The theory of representations for Boolean algebras, Transactions of the American Mathematical Society, 40 (1936) 37-111.

\bibitem{EEF32} X. Xi, C. Shen, D. Zhao, The complete Boolean algebra of regular open subsets in the real line is not sober, https://arxiv.org/abs/2608.00408, 2026.

\bibitem{EFUP62} X. Xu, W. Ji, Every countable meet-continuous lattice is Scott sober, 	arXiv:2609.19612, 2026.

\bibitem{TH62} X. Xu, C. Shen, X. Xi, D. Zhao, First countability, $\omega$-well-filtered spaces and reflections, Topology and its Applications, 279 (2020) 107255.

\bibitem{U17} X. Xu, D. Zhao, Some open problems on well-filtered spaces and sober spaces, Topology and its Applications, 301 (2021) 107540.






\end{thebibliography}
\end{document}